\documentclass[preprint,12pt,3p]{elsarticle}
\usepackage{amsmath,amsthm,amsfonts,amssymb}
\usepackage{blkarray}
\usepackage{algorithm}
\usepackage{algpseudocode}
\usepackage{hyperref}
\usepackage{xcolor}
\usepackage[normalem]{ulem}
\usepackage{sectsty}
\newtheorem{theorem}{Theorem}[section]
\newtheorem{lemma}[theorem]{Lemma}
\newtheorem{cor}{Corollary}[section]
\theoremstyle{definition}
\usepackage{hyperref}
\newtheorem{definition}{Definition}[section]
\newtheorem{rem}[subsection]{Remark}
\newtheorem{example}{Example}[section]
\newcommand{\R}{{\mathbb R}}
\begin{document}
{\bf Three-step alternating iterations for index one matrices}

Ashish Kumar Nandi, Jajati Keshari Sahoo, Debasisha Mishra\\

Department of Mathematics, BITS Pilani, K.K. Birla Goa Campus, Goa,
India\\
ashish.nandi123@gmail.com,
jksahoo@goa.bits-pilani.ac.in\\

Department of Mathematics, NIT Raipur, Raipur, India\\
dmishra@nitrr.ac.in

{\bf Abstract}\\

Iterative methods based on matrix splittings are useful in solving
large sparse linear systems.  In this direction, proper splittings
and its several extensions are used to deal with singular and
rectangular linear systems. In this article, we introduce a new
iteration scheme called three-step alternating iterations using
proper splittings and group inverses to find an approximate solution
of singular linear systems, iteratively. A preconditioned
alternating iterative scheme is also proposed to relax some
sufficient conditions and to obtain faster convergence as well. We
then show that our scheme converges faster than the existing one.
The theoretical findings are then validated numerically.

\section{Introduction}
 A real square matrix $A$ is called a {\it $Z$-matrix} if the
off-diagonal entries of $A$ are non-positive. A $Z$-matrix $A$ can
be written as $A= sI-B$, where $s \geq  0$ and $B \geq 0$. Here $B
\geq 0$ means all the entries of $B$ are non-negative. A $Z$-matrix
$A$ is called an {\it $M$-matrix} if $s \geq \rho(B)$, where
$\rho(B)$ denotes the {\it spectral radius} of $B$ and is the
maximum of the moduli of the eigenvalues of $B$. If $s > \rho(B)$,
it follows that $A^{-1}$ exists  and $A^{-1}\geq 0$. Many
interesting characterizations of nonsingular $M$-matrices can be
found in the book by Berman and Plemmons \cite{c3}. The set of
nonsingular $M$-matrices are one of the most important subclass of
monotone matrices. A real $n \times n$ matrix $A$ is called {\it
monotone}
 if $Ax \geq  0 ~\Rightarrow x \geq 0$. The book by Collatz \cite{coltz} has discussed the natural occurrence of monotone matrices in  finite difference approximation methods for certain type of partial differential equations. This class of matrices also arises in linear complementary problems in operations research, input-output production
and growth models in economics and Markov processes in probability
and statistics, to name a few.
    Singular $M$-matrices  (when $s=\rho(B)$) very often appear
in the same context as nonsingular $M$-matrices, in particular in
the study of Markov processes (see Meyer \cite{mey}). These matrices
also arise in finite difference methods for solving certain partial
differential equations such as the Neumann problem and Poisson's
equation on a sphere (see Plemmons \cite{ple}).
    The books by Berman and Plemmons \cite{c3} and Varga \cite{var} give an excellent account of many  characterizations of the notion of monotonicity to singular and rectangular matrices. In this article, we focus on the convergence of iterative methods for solving singular linear systems  using  group (generalized) inverses. This study will  help us to find an approximate solution of a singular linear system of the form
    \begin{equation}\label{eq1}
Ax = b,
\end{equation}
where $A$ is a real $n \times n$ matrix of index 1 and $x,b$ are
real $n$-vectors. For a real square matrix $A$, the {\it index of
$A$} is defined as the smallest non-negative integer $k$, which
satisfies $rank(A^{k}) = rank(A^{k+1})$. We call a {\it singular
linear system $Ax=b$ of index 1} if index of $A$ is 1. The group
(generalized) inverse  of a matrix $A\in {\R}^{n\times n}$, denoted
by $A^{\#}$ (if it exists), is the unique matrix $X$ satisfying
$A=AXA$, $X=XAX$ and $AX=XA$. For index 1 matrices, it always
exists. A group invertible matrix $A$ is called {\it group monotone}
if $A^{\#}\geq 0.$

Wei \cite{wei} showed that for a singular linear system $Ax=b$ of
index 1, the iteration scheme:
\begin{equation}\label{eq2}
x^{i+1} = U^{\#}Vx^{i}+U^{\#}b
\end{equation}
converges to $A^{\#}b$ if and only if $\rho(U^{\#}V)<1$ (see
Corollary 3.2, \cite{wei}) by using proper splitting $A=U-V$. A
splitting $A=U-V$ of
 $A\in {\R}^{n\times n}$ is called a {\it proper
splitting} \cite{c4} if $R(U)=R(A)$ and $N(U)=N(A)$, where $R(B)$
and $N(B)$ stand for the range space and the null space of a matrix
$B$, respectively.  Thereafter, he studied the convergence of the
above iteration scheme for different sub-classes of proper
splittings (see Theorem 4.1 \& 4.2, \cite{wei}).

However, the iteration scheme (\ref{eq2}) converges very slow in
many practical cases. To overcome this, several comparison results
are proposed in the literature (see  \cite{giriindex},
\cite{jenaaot}, \cite{jena}, \cite{jmd} and \cite{yim} and the
references cited therein). In  case of a matrix having many proper
splittings, comparison results are not so useful to find the best
splitting (in the sense that the iteration matrix arising from a
matrix splitting has the smallest spectral radius). To deal with
this case, we propose a three-step alternating iteration scheme by
extending the idea of  Benzi and Szyld \cite{benz} who proposed the
concept of two-step alternating iteration method.

The rest of the paper is sectioned as follows. In the next section,
we introduce our notations, definitions and some preliminary results
which are basics for defining our problem. The notion of proper
G-regular and proper G-weak regular splitting along with some
perquisite results are proved in section 3. Section 4 contains the
main results which discuss convergence criteria for the proposed
alternating iteration scheme. It also provides an algorithm for the
three-step alternating iteration scheme with a little emphasis on
preconditioning technique. The theoretical results are then
validated through computation and are shown in section 5. The last
one is about concluding remarks.

\section{ Preliminaries}
In the subsequent sections, ${\R}^n$ means an $n$-dimensional
Euclidean space while ${\R}^{n\times n}$ denotes the set of all real
square matrices of order $n$. Assume that $S$ and $T$ are
complementary subspaces of $\mathbb{R}^{n}$. Then $P_{S,T}$ is the
projection on $S$ along $T$. So, $P_{S,T}A = A$ if and only if
$R(A)\subseteq S$ and $AP_{S,T} = A$ if and only if $N(A)\supseteq
T$. We next produce the definitions of three important generalized
inverses. The {\it Drazin inverse} of a  matrix $A\in {\R}^{n\times
n}$
 is the unique solution $X\in {\R}^{n\times n}$ satisfying  the equations:
$A^k=A^kXA$, $X=XAX$ and $AX=XA$, where $k$ is the index of $A$. It
is denoted by $A^{D}$.
 When $k=1$, then Drazin inverse is said to be the {\it group inverse} of $A$. For $A \in{\R} ^{m\times n}$, the unique
matrix $Z \in {\R}^{n \times m}$ satisfying the following four
equations known as Penrose equations:
$AZA=A,~ZAZ=Z,~(AZ)^{t}=AZ~\text{and}~ (ZA)^{t}=ZA$ is called the
{\it Moore-Penrose inverse} of $A$, where $B^t$ denotes the
transpose of $B$. It always exists, and is denoted by $A^{\dag}$.
When the matrix $A$ is nonsingular, then
$A^D=A^{\#}=A^{\dag}=A^{-1}$. The criteria `index 1' for the
existence of the group inverse is also equivalent to $N(A)=N(A^2)$
or $R(A)=R(A^2)$ or $R(A)\oplus N(A)=\mathbb{R}^{n}$. A few basic
properties which will be frequently used are: $R(A) = R(A^{\#})$;
$N(A) = N(A^{\#})$; $AA^{\#} = P_{R(A),N(A)} = A^{\#}A$. In
particular, if an element $x\in R(A)$, then $x = A^{\#}Ax$.

The computation of the group inverse of an index one matrix is shown
in Algorithm \ref{alg2}, and the same method can be found in
\cite{cdmeyer}.
 \begin{algorithm}[H]
\caption{Computation of the Group Inverse}\label{alg2}
\begin{algorithmic}[1]
\Procedure{GINV}{$A$} \If {$rank(A)=rank(A^2)$} \State $r = rank(A)$
\State $Q = [B_{R(A)}~~B_{N(A)}]$ \State $P = Q^{-1}AQ$ \State
denote $C=$ Top $r\times r$ sub matrix of $P$ \State $D=
  \left[
\begin{array}{c|c}
C^{-1} & {\bf0} \\
\hline {\bf 0} & {\bf0}
\end{array}
\right]_{n\times n}$ \State\Return $A^{\#}=QDQ^{-1}$ \Else \State
``The matrix is not of index $1$" \EndIf \EndProcedure
\end{algorithmic}
\end{algorithm}

The remaining results are collected in the next two subsections.

\subsection{Non-negative matrices}
   We call $A \in {\R}^{n\times n}$ as {\it non-negative (positive)} if $A\geq 0,~(A>0)$.
  We write $B \geq C$ if $B-C \geq 0$.   The same
notation and nomenclature are also used for vectors. The next
results deal  with the non-negativity of a matrix and the spectral
radius.

\begin{theorem}[Theorem 2.20, \cite{var}]\label{2.1}
Let $A\in{\mathbb{R}^{n\times{n}}}$ and $A\geq 0$. Then\\
$(i)$ $A$ has a non-negative real eigenvalue equal to its spectral radius.\\
$(ii)$ There exists a non-negative eigenvector for its spectral
radius.
\end{theorem}

\begin{theorem}[Theorem 2.1.11, \cite{c3}]\label{2.4}
Let  $B\in{\mathbb{R}^{n\times{n}}}$, $B\geq 0$, $x\geq 0$ $(x\neq 0)$ and $\alpha$ is a positive scalar.\\
$(i)$ If $\alpha x\leq Bx$, then $\alpha\leq \rho(B)$.\\
$(ii)$ If $Bx-\alpha x\leq 0$, $x>0$, then $\rho(B)\leq \alpha$.
\end{theorem}
The last result is a special case of Theorem 3.16, \cite{var}.

\begin{theorem}\label{2.3}
Let $X\in{\mathbb{R}^{n\times{n}}}$ and $X\geq 0$. Then $\rho(X)<1$
if and only if $(I-X)^{-1}$ exists and  $(I-X)^{-1} =
\displaystyle{\sum\limits_{k=0}^{\infty}X^{k}\geq 0}$.
\end{theorem}

\subsection{Proper Splittings}
The notion of proper splitting introduced by Berman and Plemmons
\cite{c4}  plays a key role in the study of the convergence of
iterative methods to find an approximate solution of real large
singular and rectangular linear systems. It is extended to index
splitting by Wei \cite{wei} and index-proper splitting by Chen and
Chen \cite{chen} to find the approximate iterative solution of
$A^Db$ which is helpful in the study of singular differential and
difference equations (see Chapter 9, \cite{camp}).  A method of
construction of proper splitting can be found in \cite{misoam} while
its uniqueness is shown very recently in \cite{nmdm}. The result
produced below is a combination of  Theorem 5.2, \cite{c5} and
Theorem 4.1, \cite{misoam}, and is also a special case of Theorem
3.2 $\&$ 3.3, \cite{jena} and Theorem 3.1, \cite{wei} when index $1$
matrices are considered.

\begin{theorem}\label{2.5}
Let $A = U-V$ be a proper splitting of $A\in{\mathbb{R}^{n\times{n}}}$. Suppose that $A^{\#}$ exists. Then\\
$(a)$ $U^{\#}$ exists. \\
$(b)$ $AA^{\#} = UU^{\#};A^{\#}A = U^{\#}U$.\\
$(c)$ $A = U(I-U^{\#}V) =(I-VU^{\#})U$.\\
$(d)$ $I-U^{\#}V$ and $I-VU^{\#}$ are nonsingular.\\
$(e)$ $A^{\#} = {(I-U^{\#}V)}^{-1}U^{\#}= {U^{\#}(I-VU^{\#})}^{-1}$.
\end{theorem}

\section{Proper G-regular \& Proper G-weak regular Splitting}
In this section, we recall first the definition of proper G-regular
splittings and proper G-weak regular splittings, and then present
some new results for index 1 matrices. Definition 2.1 and 2.2,
\cite{jenaaot} reduce to the following two definitions, respectively
when we use the group inverse in the place of the Drazin inverse.

\begin{definition}
Let $A = U-V$ be a proper splitting of $A\in\mathbb{R}^{n\times n}.$
Then the splitting is called \textit{proper G-regular} splitting if
$U^{\#}$ exists, $U^{\#}\geq 0$ and $V\geq 0$.
\end{definition}

\begin{definition}
Let $A = U-V$ be a proper splitting of $A\in\mathbb{R}^{n\times n}.$
Then the splitting is called a \textit{proper G-weak regular}
splitting if $U^{\#}$ exists, $U^{\#}\geq 0$ and $U^{\#}V\geq 0$.
\end{definition}

 Below is an algorithm which we have used for computing proper G-weak regular splittings in this article.

\begin{algorithm}[H]
\caption{Generation of Proper G-weak regular splittings}\label{alg3}
\begin{algorithmic}[1]
\Procedure{PROP G-WEAK REG}{$A$} \State \small{Generate
$B=\{K:R(A)=R(K)~ \&~ N(A)=N(K)\}$} \While {(true)} \State $U =
\mbox{Random}(B)$ \If {($U^{\#}\geq 0~ \&~U^{\#}(U-A)\geq 0)$}
\State\Return $U$ \EndIf \EndWhile \EndProcedure
\end{algorithmic}
\end{algorithm}
The next example shows that a proper G-weak regular splitting does
not imply a proper G-regular splitting.
\begin{example}
Let $A =
  \begin{bmatrix}
    1&-1&3\\
   -1&10&-3\\
    3&-3&9
  \end{bmatrix}
  =
  \begin{bmatrix}
    2&-1&6\\
   -2&10&-6\\
    6&-3&18
  \end{bmatrix}
  -
  \begin{bmatrix}
    1&0&3\\
    {\bf -1 }&0& {\bf -3}\\
    3&0&9
  \end{bmatrix}
= U-V$. Then $R(U) = R(A)$, $N(U) = N(A)$, $U^{\#} =
\begin{bmatrix}
    0.0056&0.0056&0.0167\\
    0.0112&0.1111&0.0335\\
    0.0167&0.0167&0.05
  \end{bmatrix}
\geq 0$ and $U^{\#}V =
\begin{bmatrix}
    0.5&0&0.15\\
    0.0004&0&0.0012\\
    0.15&0&0.45
  \end{bmatrix}
\geq 0.$ Hence, the splitting $A = U-V$ is a proper G-weak regular
splitting but not a proper G-regular splitting since $V \not\geq
{0}$.
\end{example}

From the above example, it is clear that the class of proper G-weak
regular splittings contains the class of proper G-regular
splittings. We next recall convergence results for both of these
class of matrices which also characterize the notion of group
monotonicity. The first one concerns a proper G-regular splitting of
a matrix, and a particular case of Theorem 3.2 and Theorem 3.4 of
\cite{abj}.

\begin{theorem}\label{3.1}
Let $A = U-V$ be a proper G-regular splitting of $A\in
{\mathbb{R}^{n\times{n}}}$. Then  $A^{\#}\geq 0$ if and only if
$\rho(U^{\#}V)<1$.
\end{theorem}

The next one is about the convergence of proper G-weak regular
splittings. It follows from Theorem 3.8, \cite{abj} and Theorem 4.2,
\cite{wei}.

\begin{theorem}\label{3.2}
Let $A = U-V$ be a proper G-weak regular splitting of
$A\in{\mathbb{R}^{n\times{n}}}$. Then $A^{\#}\geq 0$ if and only if
$\rho(U^{\#}V)<1$.
\end{theorem}

The rate of convergence of the scheme (\ref{eq2}) depends upon
$\rho(U^{\#}V)$. Therefore,  the smaller spectral radius of the
iteration matrix yields the faster convergence rate of the iterative
scheme (\ref{eq2})  to solve the system (\ref{eq1}). The next result
helps us to choose an iteration scheme having the faster convergence
rate if $A$ has two different subclasses of proper splitting which
leads to two different iteration schemes.

\begin{theorem}\label{3.6}
Let $A = B-C$ be a proper G-weak regular splitting and $A = U-V$ be
a proper G-regular splitting of a group monotone matrix
$A\in\mathbb{R}^{n\times{n}}.$ If $B^{\#}\geq U^{\#}$, then
$\rho(B^{\#}C)\leq \rho(U^{\#}V)<1$.
\end{theorem}
\begin{proof}
 By Theorem \ref{3.1} and Theorem \ref{3.2}, we have $\rho(U^{\#}V)<1$  and  $\rho(B^{\#}C)<1$, respectively. Since $B^{\#}C\geq 0$, there exists an eigenvector $x\geq 0$ such that $x^{t}B^{\#}C = \rho(B^{\#}C)x^{t}$, by Theorem \ref{2.1}. Hence $x\in R(C^{t})\subseteq R(B^{t}) = R(A^{t})$. Now, the condition  $B^{\#}\geq U^{\#}$  yields $(I-B^{\#}C)A^{\#}\geq A^{\#}(I-VU^{\#})$ by using Theorem \ref{2.5} (e). This implies $B^{\#}CA^{\#} \leq A^{\#}VU^{\#}$. Pre-multiplying $x^{t}$ to $B^{\#}CA^{\#} \leq A^{\#}VU^{\#}$, we obtain $x^{t}B^{\#}CA^{\#} \leq x^{t}A^{\#}VU^{\#},$ i.e.,
 $\rho(B^{\#}C)x^{t}A^{\#}\leq x^{t}A^{\#}VU^{\#}.$
 Setting $x^{t}A^{\#}=z^t$ and taking transpose both sides, we get $\rho(B^{\#}C)z\leq (VU^{\#})^tz.$
 Therefore, by Theorem \ref{2.4} (i), we have  $\rho(B^{\#}C)\leq \rho(VU^{\#})^t= \rho(VU^{\#})=\rho(U^{\#}V)<1$ as $z\geq 0$ and $z\neq 0$ which is shown below by the method of contradiction. Suppose that $z=0$. Then
 $(A^{t})^{\#}x=(A^{\#})^{t}x=(x^{t}A^{\#})^t=z^t=0.$
 So $x=\displaystyle{P_{R(A^t), N(A^t)}}x=A^{t}(A^{t})^{\#}x = 0$ as $x\in R(A^{t})$ which is a contradiction.
 \end{proof}

We remark that the problems mentioned to be open in the concluding
section of \cite{misalt} can be easily now solved by using the
Moore-Penrose inverse version of the above result which is Theorem
2.8, \cite{nmdm}. The proof of the above result follows analogous
technique as in the proof of Theorem 2.8, \cite{nmdm}. However,
these ideas are completely different from \cite{els}, where the
author proved the above result in the nonsingular matrix setting.
The present proof is much simple than Elsener\cite{els}'s  one. One
may refer part (c) of a Lemma proved in section 3 of \cite{els} for
the same, and the same is produced below.

\begin{cor} \cite{els}
Let $A = B-C$ be a weak regular splitting and $A = U-V$ be a regular
splitting of  $A\in\mathbb{R}^{n\times{n}}$. If $B^{-1}\geq U^{-1}$
and $A^{-1}\geq 0$, then $\rho(B^{-1}C)\leq \rho(U^{-1}V)<1$.
\end{cor}

\section{Three-step  Alternating Iterations}

Throughout this section, we consider the co-efficient matrix $A$ in (\ref{eq1}) as of index 1 unless otherwise mentioned. Let $A = K-L = U-V = X-Y$ be three proper splittings of $A\in\mathbb{R}^{n\times{n}}$. Let us consider the following  iterative schemes:\\
\begin{equation}\label{eq3}
x^{k+1/3} = K^{\#}Lx^{k}+K^{\#}b
\end{equation}
\begin{equation}\label{eq4}
x^{k+1/2} = U^{\#}Vx^{k+1/3}+U^{\#}b
\end{equation}
\begin{equation}\label{eq5}
x^{k+1} = X^{\#}Yx^{k+1/2}+X^{\#}b
\end{equation}
to introduce the  three-step alternating iteration scheme.    We form a single iteration  scheme by eliminating  $x^{k+1/3}$ and $x^{k+1/2}$ from (\ref{eq3}), (\ref{eq4}) and (\ref{eq5}) to do computation. So, we get\\
\begin{equation}\label{eq6}
x^{k+1} =
X^{\#}YU^{\#}VK^{\#}Lx^{k}+X^{\#}(YU^{\#}VK^{\#}+YU^{\#}+I)b,\hspace{0.1in}{k
= 0,1,2,\cdots},
\end{equation}
where $H = X^{\#}YU^{\#}VK^{\#}L$ is the iteration matrix of the new
iterative scheme (\ref{eq6}) called as the {\it three-step
alternating iteration scheme}. The convergence of the individual
splitting need not imply the convergence of the three-step
alternating iteration scheme (\ref{eq6}) which is shown in the
following example.
\begin{example}\label{ex1}
Let $A =
  \begin{bmatrix}
    4&4&10\\
    7&-29&31\\
    -1&11&-7
  \end{bmatrix},
     K =
  \begin{bmatrix}
    17.6&0.8&50.3\\
    -41.2&-41.8&-102.775\\
    19.6&14.2&51.025
  \end{bmatrix},$
  \begin{equation*}
  U =
  \begin{bmatrix}
    2.4&15.2&1.2\\
    18.6&-31&65.1\\
    -5.4&15.4&-21.3
  \end{bmatrix},
  X =
  \begin{bmatrix}
    5.6&2.4&15.2\\
   -91&-111&-220\\
    32.2&37.8&78.4
  \end{bmatrix}.
  \end{equation*}
  Here $A = K-L = U-V = X-Y$ are three proper splittings. Also $\rho(K^{\#}L) = 0.6835<1$, $\rho(U^{\#}V) = 0.5957<1$, $\rho(X^{\#}Y) = 0.8452<1$ but $\rho(H) = 1.3579\not<1$.
  \end{example}
 Note that all the computations in this paper are made in fractions but for the presentation point of view, we have rounded to 4 decimal places. So, there might be a little rounding error.
The algorithm for the three-step alternating iterations is produced
next, and the same is also used in Example \ref{ex1}.


Example \ref{ex1} motivates further to study the
 convergence criteria of the three-step alternating iterations, and the next result is in the same direction.

\begin{theorem}\label{conv}
If $A = K-L = U-V = X-Y$ are three proper G-weak regular splittings
of a group monotone matrix $A$, then $\rho(X^{\#}YU^{\#}VK^{\#}L)
<1$.
\end{theorem}
\begin{proof}
 We have $H=X^{\#}YU^{\#}VK^{\#}L\geq 0$ as $A = K-L = U-V = X-Y$ are three proper G-weak regular splittings. By Theorem \ref{2.5} (b), $A^{\#}A = K^{\#}K = U^{\#}U = X^{\#}X.$ Since
 \begin{eqnarray*}\label{eqnp}
\nonumber 
X^{\#}AU^{\#}AK^{\#}A&=&  X^{\#}(X-Y)U^{\#}(U-V)K^{\#}A\\
\nonumber
  &=& X^{\#}XU^{\#}UK^{\#}A-X^{\#}XU^{\#}VK^{\#}A-X^{\#}YU^{\#}UK^{\#}A+X^{\#}YU^{\#}VK^{\#}A\\
   &=& K^{\#}A-U^{\#}VK^{\#}A-X^{\#}YK^{\#}A+X^{\#}YU^{\#}VK^{\#}A,
 \end{eqnarray*} so the iteration matrix $H$ is expressed as
\begin{eqnarray}\label{eqnh}
 \nonumber
 H &=& X^{\#}YU^{\#}VK^{\#}L\\
 \nonumber
   &=& X^{\#}(X-A)U^{\#}(U-A)K^{\#}(K-A)\\
 \nonumber
   &=& (U^{\#}U-U^{\#}A-X^{\#}A+X^{\#}AU^{\#}A)(K^{\#}K-K^{\#}A) \\
 \nonumber
   &=& A^{\#}A-U^{\#}A-X^{\#}A+X^{\#}AU^{\#}A-K^{\#}A\\
   \nonumber
   & & ~~+U^{\#}AK^{\#}A+X^{\#}AK^{\#}A-X^{\#}AU^{\#}AK^{\#}A.\\
   \nonumber
 &=& A^{\#}A-U^{\#}A-X^{\#}A+X^{\#}AU^{\#}A-K^{\#}A+U^{\#}AK^{\#}A+X^{\#}AK^{\#}A\\
 \nonumber
   & &~~-K^{\#}A+U^{\#}VK^{\#}A+X^{\#}YK^{\#}A-X^{\#}YU^{\#}VK^{\#}A.
 \end{eqnarray}
 This implies
 \begin{eqnarray*}
   HA^{\#} &=& A^{\#}-U^{\#}-X^{\#}+X^{\#}AU^{\#}-K^{\#}+U^{\#}AK^{\#}\\
   && ~+X^{\#}AK^{\#}-K^{\#}+U^{\#}VK^{\#}+X^{\#}YK^{\#}-X^{\#}YU^{\#}VK^{\#},
   \end{eqnarray*}
   and hence
\begin{eqnarray*}
  (I-H)A^{\#}&=&U^{\#}+X^{\#}-X^{\#}AU^{\#}+K^{\#}-U^{\#}AK^{\#}-X^{\#}AK^{\#}\\
 &&~~+K^{\#}-U^{\#}VK^{\#}-X^{\#}YK^{\#}+
  X^{\#}YU^{\#}VK^{\#}\\
  &=& U^{\#}+X^{\#}-U^{\#}(A+V)K^{\#}-X^{\#}(A+Y)K^{\#}-X^{\#}AU^{\#}\\
  &&~~+K^{\#}+X^{\#}YU^{\#}VK^{\#}+K^{\#}\\
   &=& X^{\#}XU^{\#}-X^{\#}AU^{\#}+X^{\#}+X^{\#}YU^{\#}VK^{\#}\\
   & =& X^{\#}YU^{\#}+X^{\#}+X^{\#}YU^{\#}VK^{\#}\geq 0.
\end{eqnarray*}
Now, $0 \leq (I+H+H^{2}+\cdots+H^{m})(I-H)A^{\#}\leq A^{\#}$ for
each non-negative integer $m$. Therefore, the partial sums of the
series $\displaystyle{\sum\limits_{m=0}^{\infty}H^{m}}$ remain
uniformly bounded in norm. Hence $\rho(H) =
\rho(X^{\#}YU^{\#}VK^{\#}L) <1$.
\end{proof}

We have the following result in  case of nonsingular $A$.

\begin{cor}\label{convcor}
 If $A = K-L = U-V = X-Y$ are three weak regular splittings of a  monotone matrix $A$, then $\rho(X^{-1}YU^{-1}VK^{-1}L) <1$.
\end{cor}

The above one extends the convergence criteria of two-step
alternating iteration scheme proved by Benzi and Szyld in the
 first part of Theorem 3.2, \cite{benz}. The same is produced next as a corollary.

\begin{cor}(Theorem 3.2, \cite{benz})
 If $A  = U-V = X-Y$ are two weak regular splittings of a  monotone matrix $A$, then $\rho(X^{-1}YU^{-1}V) <1$.
\end{cor}
The next example shows that the converse of  Theorem \ref{conv} is
not  true.
\begin{example}
Let $A =
  \begin{bmatrix}
    -11&4&15\\
    12&2&9\\
    23&-2&-6
  \end{bmatrix}.$ Now

\begin{eqnarray*}
  A &=&
  \begin{bmatrix}
    -33.5&20&76.8429\\
    62&10&45.1714\\
   95.5&-10&-31.6714
  \end{bmatrix}
  -
  \begin{bmatrix}
    -22.5&16&61.8429\\
    50&8&36.1714\\
    72.5&-8&-25.6714
  \end{bmatrix}
   = K-L \\
   &=&
  \begin{bmatrix}
    -58&53&206.271\\
    41&-26&-100.114\\
    99&-79&-306.386
  \end{bmatrix}
  -
  \begin{bmatrix}
    -47&49&191.271\\
    29&-28&-109.114\\
    76&-77&-300.386
  \end{bmatrix}=U-V \\
   &=& \begin{bmatrix}
    -53&39.5&152.893\\
    61&-24&-90.4286\\
    114&-63.5&-243.321
  \end{bmatrix}
  -
  \begin{bmatrix}
    -42&35.5&137.893\\
     49&-26&-99.4286\\
     91&-61.5&-237.321
  \end{bmatrix}=X-Y
\end{eqnarray*}
 are three proper splittings of $A$. Then $\rho(H) = \rho(X^{\#}YU^{\#}VK^{\#}L) = 0.4938<1$. Now, the individual splitting have the following property:
\begin{eqnarray*}
  K^{\#} &=& \begin{bmatrix}
     0.0029&0.0025&0.0101\\
     0.0138&0.0026&0.0117\\
     0.0109&0.0002&0.0016
  \end{bmatrix}
  \geq 0
\mbox{ and~} K^{\#}L =
\begin{bmatrix}
     0.7883&-0.0140&0.0118\\
     0.6658&0.1493&0.6520\\
    -0.1224&0.1633&0.6402
  \end{bmatrix}
  \not\geq 0,\\
  & &\\
  U^{\#} &=& \begin{bmatrix}
     0.1041&0.0142&0.0654\\
     0.1100&0.0144&0.0667\\
     0.0059&0.0002&0.0013
  \end{bmatrix}
  \geq 0
\mbox{ and~} U^{\#}V =
\begin{bmatrix}
     0.4884&-0.3314&-1.2792\\
     0.3166&-0.1488&-0.5661\\
     -0.1719&0.1826&0.7131
  \end{bmatrix}
\not\geq {0},\\
& &\\
  X^{\#} &=& \begin{bmatrix}
     0.0490&0.0061&0.0284\\
     0.0577&0.0064&0.0305\\
     0.0087&0.0003&0.0021
  \end{bmatrix}
  \geq 0
\mbox{ and~} X^{\#}Y =
\begin{bmatrix}
     0.8291&-0.1687&-0.6011\\
     0.6690&0.0040&0.0734\\
     -0.1601&0.1727&0.6745
     \end{bmatrix}
\not\geq {0}.
\end{eqnarray*}
Hence $A = K-L = U-V = X-Y$ are not G-weak regular splittings of
$A$.
\end{example}

 The following result shows that the iteration matrix $H$ of the three-step alternating iterations induces a unique proper G-weak regular splitting.

\begin{theorem}\label{ind}
Let $A = K-L = U-V = X-Y$ be three proper G-weak regular splittings
of a group monotone matrix $A.$ If $R(A) = R(K+X-A+YU^{\#}L)$ and
$N(A) = N(K+X-A+YU^{\#}L),$ then there exists a unique  proper
G-weak regular splitting $A = B-C$ induced by $H$ with $B =
K(K+X-A+YU^{\#}L)^{\#}X.$
\end{theorem}
\begin{proof}
It is specified that  $A = K-L = U-V = X-Y$ are three proper
splittings of $A$. So, by Theorem \ref{2.5} (b), $A^{\#}A = K^{\#}K
= U^{\#}U = X^{\#}X$ and $AA^{\#} = KK^{\#} = UU^{\#} = XX^{\#}.$
Equation (\ref{eq6}) yields
\begin{eqnarray*}
 B^{\#} &=& X^{\#}(YU^{\#}VK^{\#}+YU^{\#}+I) \\
   &=& X^{\#}(X-A)U^{\#}(U-A)K^{\#}+X^{\#}(X-A)U^{\#}+X^{\#} \\
  &=& X^{\#}XU^{\#}UK^{\#}-X^{\#}XU^{\#}AK^{\#}-X^{\#}AU^{\#}UK^{\#} \\
   &&~~ +X^{\#}AU^{\#}AK^{\#}+X^{\#}XU^{\#}-X^{\#}AU^{\#}+X^{\#}XX^{\#}\\
  &=& X^{\#}XK^{\#}-X^{\#}XU^{\#}AK^{\#}-X^{\#}AK^{\#}+X^{\#}AU^{\#}AK^{\#}\\
   &&~~+X^{\#}XU^{\#}KK^{\#}-X^{\#}AU^{\#}KK^{\#}+
X^{\#}KK^{\#}\\
   &=&X^{\#}(X-XU^{\#}A-A+AU^{\#}A+XU^{\#}K-AU^{\#}K+K)K^{\#}.
\end{eqnarray*}
On further simplification of $AU^{\#}A-XU^{\#}A-AU^{\#}K+XU^{\#}K$,
we get $AU^{\#}A-XU^{\#}A-AU^{\#}K+XU^{\#}K=(A-X)(U^{\#}A-U^{\#}K)
   = YU^{\#}L.$
Therefore,
\begin{equation}\label{equation177}
B^{\#} = X^{\#}(K+X-A+YU^{\#}L)K^{\#}.
\end{equation}
Since $R(K+X-A+YU^{\#}L) = R(A)$ and $N(K+X-A+YU^{\#}L) = N(A)$, we
have
$(K+X-A+YU^{\#}L)^{\#}(K+X-A+YU^{\#}L)=(K+X-A+YU^{\#}L)(K+X-A+YU^{\#}L)^{\#}.$
 Let $G = K(K+X-A+YU^{\#}L)^{\#}X$, then $B^{\#}GB^{\#} = X^{\#}(K+X-A+YU^{\#}L)K^{\#}K(K+X-A+YU^{\#}L)^{\#}XX^{\#}(K+X-A+YU^{\#}L)K^{\#} = B^{\#}$ and $GB^{\#}G = K(K+X-A+YU^{\#}L)^{\#}XX^{\#}(K+X-A+YU^{\#}L)K^{\#}K(K+X-A+YU^{\#}L)^{\#}X = G$. Also $B^{\#}G = X^{\#}(K+X-A+YU^{\#}L)K^{\#}K(K+X-A+YU^{\#}L)^{\#}X = X^{\#}X = XX^{\#} = K(K+X-A+YU^{\#}L)^{\#}XX^{\#}(K+X-A+YU^{\#}L)K^{\#} = GB^{\#}$. Hence $G = (B^{\#})^{\#} = B$.\\
    Next, we prove that $A = B-C$ is a proper splitting. First, we show that $N(A) = N(B)$. Clearly, $N(X)\subseteq N(B)$ since $B = K(K+X-A+YU^{\#}L)^{\#}X$. Let $Bx = 0$ which implies    $    K(K+X-A+YU^{\#}L)^{\#}Xx = 0.$
    Pre-multiplying  by $K^{\#}$,  we get     $(K+X-A+YU^{\#}L)^{\#}Xx = 0.$

    Again, pre-multiplying the last equation by $ K+X-A+YU^{\#}L$, we have $Xx = 0$. So $N(B)\subseteq N(X)$. Hence $ N(A) = N(B)$. From (\ref{equation177}), we have
    \begin{eqnarray*}
      B^{\#} &=& X^{\#}(K+X-A+YU^{\#}L)K^{\#} \\
       &=& X^{\#}+K^{\#}-X^{\#}AK^{\#}+X^{\#}(X-A)U^{\#}(K-A)K^{\#} \\
       &=& X^{\#}+K^{\#}-X^{\#}AK^{\#}+U^{\#}-U^{\#}AK^{\#}-X^{\#}AU^{\#}+X^{\#}AU^{\#}AK^{\#} \\
       &=& A^{\#}-(A^{\#}-U^{\#}-X^{\#}+X^{\#}AU^{\#}-K^{\#}+U^{\#}AK^{\#}+X^{\#}AK^{\#}-X^{\#}AU^{\#}AK^{\#})\\
       &=& A^{\#}-(A^{\#}AA^{\#}-U^{\#}AA^{\#}-X^{\#}AA^{\#}+X^{\#}AU^{\#}AA^{\#}-K^{\#}AA^{\#}\\
       &&~~+U^{\#}AK^{\#}AA^{\#}+
    X^{\#}AK^{\#}AA^{\#}-X^{\#}AU^{\#}AK^{\#}AA^{\#})\\
    &=& A^{\#}-(U^{\#}U-U^{\#}A-X^{\#}A+X^{\#}AU^{\#}A)(K^{\#}K-K^{\#}A)A^{\#}\\
    &=& A^{\#}-X^{\#}YU^{\#}VK^{\#}LA^{\#} = (I-H)A^{\#}.
    \end{eqnarray*}
     But, we have  $\rho(H)<1$ by Theorem \ref{conv}.  So, $I-H$ is nonsingular by Theorem \ref{2.3}. Let $G_1 = B = A(I-H)^{-1}$. Now $B^{\#}G_1B^{\#} = (I-H)A^{\#}A(I-H)^{-1}(I-H)A^{\#} = B^{\#}$. Similarly, $G_1B^{\#}G_1 = G_1$. Again, $B^{\#}G_1 = (I-H)A^{\#}A(I-H)^{-1} = (A^{\#}A-H)(I-H)^{-1} = (A^{\#}A-A^{\#}AH)(I-H)^{-1} = A^{\#}A = AA^{\#} = G_1B^{\#}$. Therefore, $B = A(I-H)^{-1}$. Hence $A=B(I-H)$, and thus $R(A) = R(B)$. Therefore, $A = B-C$ is a proper splitting. Next, to show uniqueness of proper splitting $A = B-C$. Suppose there exists another induced splitting $A = B_1-C_1$ such that $H = B_1^{\#}C_1$. Then $B_1H = B_1B_1^{\#}C_1 = C_1 = B_1-A$. So, we get $B_1(I-H) = A$ and thus $B_1 = A(I-H)^{-1} = B$. Finally, $B^{\#} = X^{\#}(YU^{\#}VK^{\#}+YU^{\#}+I) = X^{\#}YU^{\#}VK^{\#}+X^{\#}YU^{\#}+X^{\#}\geq 0$ and $B^{\#}C = X^{\#}YU^{\#}VK^{\#}L\geq 0$. Therefore, $A = B-C$ is a unique proper G-weak regular splitting.
\end{proof}

The above result in  case of a nonsingular monotone matrix is stated
by the following corollary.

\begin{cor}\label{indcor}
Let $A = K-L = U-V = X-Y$ be three weak regular splittings of a
monotone matrix $A.$  Then there exists a unique  weak regular
splitting $A = B-C$ induced by $H$ with $B =
K(K+X-A+YU^{-1}L)^{-1}X.$
\end{cor}

We also remark that this extends jointly Theorem 3.2 and 3.4 of
\cite{benz}. To support Theorem \ref{ind}, we have the following
example.

\begin{example}
Let $A =
  \begin{bmatrix}
    9&-3&6\\
   -3&5&-2\\
   6&-2&4
  \end{bmatrix}
  ,~K =
  \begin{bmatrix}
    9.9&-3.3&6.6\\
    -3.3&5.5&-2.2\\
    6.6&-2.2&4.4
  \end{bmatrix}
  ,~U =
  \begin{bmatrix}
    13.5&-4.5&9\\
   -4.5&7.5&-3\\
    9&-3&6
  \end{bmatrix}
   ,\\
   \\X =
  \begin{bmatrix}
   12.6&-4.2&8.4\\
    -4.2&7&-2.8\\
    8.4&-2.8&5.6
  \end{bmatrix}
$. Now $A^{\#} =
\begin{bmatrix}
     0.0666&0.0577&0.0444\\
     0.0577&0.2500&0.0385\\
     0.0444&0.0385&0.0296
  \end{bmatrix}
  \geq 0
$ ,\\
\begin{equation*}
  K^{\#} =
\begin{bmatrix}
     0.0605&0.0524&0.0403\\
     0.0524&0.2273&0.0350\\
     0.0403&0.0350&0.0269
  \end{bmatrix}
  \geq 0
 ,~
K^{\#}L =
\begin{bmatrix}
     0.0629&0&0.0420\\
     0&0.0909&0\\
     0.0420&0&0.0280
  \end{bmatrix}
  \geq 0
,\\
\end{equation*}

\begin{equation*}
 U^{\#} =
\begin{bmatrix}
     0.0444&0.0385&0.0296\\
     0.0385&0.1667&0.0256\\
     0.0296&0.0256&0.0197
  \end{bmatrix}
  \geq 0
,~ U^{\#}V =
\begin{bmatrix}
     0.2308&0&0.1538\\
     0&0.3333&0\\
     0.1538&0&0.1026
  \end{bmatrix}
  \geq 0
,~
\end{equation*}
 \begin{equation*}
  X^{\#} =
\begin{bmatrix}
     0.0475&0.0412&0.0317\\
     0.0412&0.1786&0.0275\\
     0.0317&0.0275&0.0211
  \end{bmatrix}
  \geq 0
 \mbox{ and } X^{\#}Y =
\begin{bmatrix}
     0.1978&0&0.1319\\
     0&0.2857&0\\
     0.1319&0&0.0879
  \end{bmatrix}
  \geq 0.
 \end{equation*}
Then $A = K-L = U-V = X-Y$ are three proper G-weak regular
splittings of $A$. Also $R(K+X-A+YU^{\#}L) = R(A)$ and
$N(K+X-A+YU^{\#}L) = N(A).$ So,
\begin{equation*}
  B = K(K+X-A+YU^{\#}L)^{\#}X =
  \begin{bmatrix}
    9.0786&-3.0262&6.0524\\
   -3.0262&5.0437&-2.0175\\
    6.0524&-2.0175&4.0349
  \end{bmatrix}
\end{equation*}

   and
   \begin{equation*}
    C = B-A =
  \begin{bmatrix}
    0.0786&-0.0262&0.0524\\
    -0.0262&0.0437&-0.0175\\
    0.0524&-0.0175&0.0349
  \end{bmatrix}.
     \end{equation*}
     Now
 $B^{\#} =
\begin{bmatrix}
     0.0660&0.0572&0.0440\\
     0.0572&0.2478&0.0381\\
     0.0440&0.0381&0.0293
  \end{bmatrix}
  \geq 0
$ and $B^{\#}C =
\begin{bmatrix}
     0.0060&0&0.0040\\
     0&0.0087&0\\
     0.0040&0&0.0027
  \end{bmatrix}
  \geq 0
$.
\end{example}
\noindent Therefore, the splitting $A = B-C$ induced by $H$ is a
proper G-weak regular splitting.

The next result confirms that the proposed alternating iterative
scheme converges faster than (\ref{eq2}) under suitable assumptions.

\begin{theorem}\label{comp}
Suppose $A = K-L = U-V = X-Y$ are three proper G-regular splittings
of a group monotone matrix $A$ with $R(A) = R(K+X-A+YU^{\#}L)$ and
$N(A) = N(K+X-A+YU^{\#}L)$.  Then $\rho(H)\leq
min\{\rho(K^{\#}L),\rho(U^{\#}V),\rho(X^{\#}Y)\}<1$.
\end{theorem}
\begin{proof}
By Theorem \ref{ind}, $A = B-C$ is a proper G-weak regular splitting
induced by $H$, and from (\ref{eq6}),  $$B^{\#} =
X^{\#}(YU^{\#}VK^{\#}+YU^{\#}+I) =
X^{\#}YU^{\#}VK^{\#}+X^{\#}YU^{\#}+X^{\#} \geq X^{\#}.$$ Again,
\begin{eqnarray*}
         B^{\#} &=& X^{\#}YU^{\#}VK^{\#}+X^{\#}YU^{\#}+X^{\#} \\
       &=& X^{\#}YU^{\#}VK^{\#}+X^{\#}XU^{\#}-X^{\#}AU^{\#}+X^{\#}UU^{\#} \\
       &=& X^{\#}YU^{\#}VK^{\#}+U^{\#}+X^{\#}(U-A)U^{\#}\\
       &=& X^{\#}YU^{\#}VK^{\#}+U^{\#}+X^{\#}VU^{\#} \geq U^{\#}.
       \end{eqnarray*}
Also,
\begin{eqnarray*}
         B^{\#} &=& X^{\#}(K+X-A+YU^{\#}L)K^{\#} \\
       &=& X^{\#}KK^{\#}+X^{\#}XK^{\#}-X^{\#}(K-L)K^{\#}+X^{\#}YU^{\#}LK^{\#}\\
       &=& X^{\#}+K^{\#}-X^{\#}KK^{\#}+X^{\#}LK^{\#}+X^{\#}YU^{\#}LK^{\#} \\
       &=& X^{\#}+K^{\#}-X^{\#}+X^{\#}LK^{\#}+X^{\#}YU^{\#}LK^{\#}\\
       &=& K^{\#}+X^{\#}LK^{\#}+X^{\#}YU^{\#}LK^{\#} \geq K^{\#}. \\
     \end{eqnarray*}
Applying Theorem \ref{3.6} to the pair of the splittings $A = B-C$
and $A = K-L$,
 $A = B-C$ and $A = U-V$, and $A = B-C$ and $A = X-Y$, we have $\rho(H)\leq \rho(K^{\#}L)<1$,
$\rho(H)\leq \rho(U^{\#}V)<1$ and $\rho(H)\leq \rho(X^{\#}Y)<1,$
respectively. Therefore, $\rho(H)\leq
min\{\rho(K^{\#}L),\rho(U^{\#}V),\rho(X^{\#}Y)\}<1$.
\end{proof}

The result below is the case when $A$ is nonsingular.

\begin{cor}\label{icomp}
Let $A = K-L = U-V = X-Y$ be three regular splittings of a monotone
matrix $A$. Then $\rho(H)\leq
min\{\rho(K^{-1}L),\rho(U^{-1}V),\rho(X^{-1}Y)\}<1$.
\end{cor}

Again, we have the following corollary when two splitting are
considered, and is proved in \cite{benz}.

\begin{cor}(Theorem 4.1, \cite{benz})
Let $A = U-V = X-Y$ be two regular splittings of a monotone matrix
$A$. Then $\rho(H)\leq min\{\rho(U^{-1}V),\rho(X^{-1}Y)\}<1$.
\end{cor}

The converse of  Theorem \ref{comp} does not hold. The next example
justifies the claim.

\begin{example}
Let $A =
  \begin{bmatrix}
   -1&0&-3\\
    0&1&2\\
    0&2&4
  \end{bmatrix}
  ,~K =
  \begin{bmatrix}
    -2&0&-6\\
    0&1&2\\
    0&2&4
  \end{bmatrix}
  ,~U =
  \begin{bmatrix}
    -3&0&-9\\
    0&1&2\\
    0&2&4
  \end{bmatrix}
   ,\\ X =
  \begin{bmatrix}
   -4&0&-12\\
    0&1&2\\
    0&2&4
  \end{bmatrix}
$. Then $A = K-L = U-V = X-Y$ are three proper splittings with
$\rho(H) = \rho(K^{\#}LU^{\#}VX^{\#}Y) = 0.25<1.$ But $A = K-L = U-V
= X-Y$ are not proper G-regular splittings as $K^{\#} =
\begin{bmatrix}
    -0.5000&0.3600&-0.7800\\
     0&0.0400&0.0800\\
     0&0.0800&0.1600
  \end{bmatrix}
  \not\geq 0
$, $L =
\begin{bmatrix}
     -1&0&-3\\
      0&0&0\\
      0&0&0
  \end{bmatrix}
\not\geq {0} $, $U^{\#} =
\begin{bmatrix}
    -0.3333&0.1600&-0.6800\\
     0&0.0400&0.0800\\
     0&0.0800&0.1600
  \end{bmatrix}
  \not\geq 0
$, $V =
\begin{bmatrix}
     -2&0&-6\\
      0&0&0\\
      0&0&0
  \end{bmatrix}
\not\geq {0} $, $X^{\#} =
\begin{bmatrix}
    -0.2500&0.0600&-0.6300\\
     0&0.0400&0.0800\\
     0&0.0800&0.1600
  \end{bmatrix}
  \not\geq 0
$ and $Y =
\begin{bmatrix}
    -3&0&-9\\
      0&0&0\\
      0&0&0
  \end{bmatrix}
\not\geq {0}$.
\end{example}

The next example shows that the condition of G-regular cannot be
dropped.
\begin{example}
Let \begin{eqnarray*}
 A =
  \begin{bmatrix}
    25&-6&1\\
    -7&4&0\\
    4&6&1
  \end{bmatrix}
  , ~A^{\#}=
  \begin{bmatrix}
    0.0428&0.0200&0.0054\\
    0.0539&0.2218&0.0305\\
    0.2044&0.6854&0.0968
  \end{bmatrix}
 \geq 0,
 \end{eqnarray*}

  \begin{eqnarray*}
  K =
  \begin{bmatrix}
    -8.75&30.5&3.0776\\
     8.25&-15.5&-1.3017\\
     16&-16&-0.8276
  \end{bmatrix}
   ,~U =
  \begin{bmatrix}
    -17.75&43&3.9655\\
    14.25&-19&-1.3103\\
    25&-14&0.0345
  \end{bmatrix}
  ,
  \end{eqnarray*}
  \begin{eqnarray*}
    X =
  \begin{bmatrix}
    -58.5&64.75&3.7802\\
     24.5&-11.75&0.2716\\
     15&29.5&4.5948
  \end{bmatrix}
    . ~\mbox{Now}~X^{\#} =
\begin{bmatrix}
     0.0911&0.1619&0.0258\\
     0.0370&0.0195&0.0049\\
     0.2020&0.2203&0.0405
  \end{bmatrix},
    \end{eqnarray*}
    \begin{eqnarray*}
 K^{\#} =
\begin{bmatrix}
     0.0436&0.0956&0.0145\\
     0.0283&0.0214&0.0045\\
     0.1285&0.1597&0.0281
  \end{bmatrix}
  \mbox{and}~U^{\#} =
\begin{bmatrix}
     0.0031&0.0430&0.0054\\
     0.0147&0.0290&0.0045\\
     0.0473&0.1299&0.0189
  \end{bmatrix}.
  \end{eqnarray*}
  Then $A = K-L = U-V = X-Y$ are three proper splittings with $R(A) = R(K+X-A+YU^{\#}L)$ and $N(A) = N(K+X-A+YU^{\#}L)$. But $A = K-L = U-V = X-Y$ are not G-regular splittings as
 \begin{eqnarray*}
 L =
\begin{bmatrix}
     -33.75&36.5&2.0776\\
      15.25&-19.5&-1.3017\\
      12&-22&-1.8276
  \end{bmatrix}
\not\geq {0} ,~ V =
\begin{bmatrix}
     -42.75&49&2.9655\\
      21.25&-23&-1.3103\\
      21&-20&-0.9655
  \end{bmatrix}
\not\geq {0}
\end{eqnarray*}
and
\begin{eqnarray*}
Y =
\begin{bmatrix}
     -83.5&70.75&2.7802\\
      31.5&-15.75&0.2716\\
      11&23.5&3.5948
  \end{bmatrix}
\not\geq {0} .\mbox{~Therefore,~}
\end{eqnarray*}
$$\rho(H) = 1.7746\not\leq \min\{\rho(K^{\#}L) = 1.2987, \rho(U^{\#}V) = 1.2530, \rho(X^{\#}Y) = 1.2975\}\not\leq 1.$$
\end{example}

Theorem \ref{conv} shows that the assumption of  group monotonicity
of $A$ guarantees the convergence of the three-step alternating
iteration scheme. If we drop this assumption,  then the proposed
theory may fail. To overcome this, the concept of a preconditioned
matrix is introduced next. In such a case, we consider the following
system \begin{equation}\label{peqn} QAx=Qb
\end{equation}
 where  $Q$ is a nonsingular matrix called {\it preconditioned matrix}. Milaszewicz \cite{mil}, used the iteration matrix $T$ which is irreducible and non-negative to improve the convergence rate of the Gauss-Seidel and the Jacobi method. Gunawardena {\it et al.} \cite{gun} proposed the preconditioned matrix $P_c = I+S$, (where $S$ is the matrix shown in remark 3.3 \cite{gun}). Kohno {\it et al.} \cite{koh} and Kotakemori {\it et al.} \cite{kot} extended the upper triangular approach by considering a parametric preconditioned matrix $P_c = I+S(\alpha)$ to obtain faster convergence in the iterative schemes which used for solving consistent linear systems. In  case of a singular linear system, we discuss the system (\ref{peqn}) converges to  $A^{\#}b$  under suitable choice of $Q$. The iterative scheme of the modified system (\ref{peqn}) is defined by,\\
\begin{equation}\label{precon}
x^{k+1} = K_{q}^{\#}L_{q}x^{k}+K_{q}^{\#}Qb,
\end{equation}
where $QA = K_{q}-L_{q}$ be a proper splitting of the matrix
$QA\in\mathbb{R}^{n\times{n}}$, will converge to $A^{\#}b$ for any
initial guess $x^{0}$ if and only if $\rho(K_q^{\#}L_q)<1$.

  Next, we discuss the existence of preconditioned matrix for some particular cases as well as the convergence of the iterative scheme for proper G-weak regular splittings.
\begin{lemma}\label{4.4}
If there exists a nonsingular matrix $Q \in\mathbb{R}^{n\times{n}}$
such that $QA = AQ,$ then $(QA)^{\#} = A^{\#}Q^{-1} = Q^{-1}A^{\#}.$
\end{lemma}
\begin{proof}
The assumption $QA = AQ$ yields $A = Q^{-1}AQ.$ Now $A^{\#} =
Q^{-1}A^{\#}Q$ is the group inverse of $A$ which can be easily
verified by the definition of group inverse. Pre-multiplying $Q$ in
$A^{\#} = Q^{-1}A^{\#}Q$ we obtain $QA^{\#} = A^{\#}Q$. Let $B = QA$
and $X = A^{\#}Q^{-1}$. By the definition of the group inverse: $BXB
= QAA^{\#}Q^{-1}QA = B,$ $XBX = A^{\#}Q^{-1}QAA^{\#}Q^{-1} = X$ and
$BX = QAA^{\#}Q^{-1} = AQA^{\#}Q^{-1} = AA^{\#}QQ^{-1} = A^{\#}A =
A^{\#}Q^{-1}QA = XB$. Now post-multiplying $Q^{-1}$ in $A^{\#} =
Q^{-1}A^{\#}Q$ we have $A^{\#}Q^{-1} = Q^{-1}A^{\#}.$
\end{proof}
\rem For  $A^{\#}\leq 0,$ if we choose $Q=-cI,(c>0)$ and for
$A^{\#}\geq 0,$ if we choose $Q=cI,(c>0)$  then Lemma \ref{4.4}
along with $(QA)^{\#}\geq 0$ holds.

\begin{lemma}\label{pconv}
Let $A^{\#}\ngeq 0.$ If there exists a nonsingular matrix
$Q\in\mathbb{R}^{n\times n}$ such that $QA = AQ,$  $A^{\#}Q^{-1}\geq
0,$ and $QA=K_q-L_q$ is a proper G-weak regular splittings of  $QA$,
then the iterative scheme (\ref{precon}) converges to $A^{\#}b.$
\end{lemma}
\begin{proof}
Since $QA = AQ$ and $A^{\#}Q^{-1}\geq 0$, we obtain
$(QA)^{\#}=A^{\#}Q^{-1}\geq 0$. As $QA=K_q-L_q$ is a proper G-weak
regular splittings of a group monotone matrix  $QA.$ Then, the
iterative scheme (\ref{precon}) converges to
$(QA)^{\#}Qb=(QA)^{\#}Qb=A^{\#}Q^{-1}Qb=A^{\#}b$, by Theorem
\ref{3.2}.
\end{proof}
\rem If $A^\#$ is neither non-positive nor non-negative, i.e., some
elements of  $A^\#$ are positive and some are negative, then the
construction of such $Q$ seems to be an open problem.

\rem  Under the same assumptions as in Lemma \ref{pconv}, and if
$QA=K_q-L_q=U_q-V_q=X_q-Y_q$ are three proper G-weak regular
splittings of $QA.$ Then, by Theorem \ref{conv} the alternating
iterative scheme generated from the splittings of $QA$, i.e.,
\begin{equation}\label{palt}
x^{k+1} =
X_q^{\#}Y_qU_q^{\#}V_qK_q^{\#}L_qx^{k}+X_q^{\#}(Y_qU_q^{\#}V_qK_q^{\#}+Y_qU_q^{\#}+I)b
\end{equation}
will converge to $A^{\#}b$, for any initial value $x^0.$ The
numerical implementation of the iterative scheme (\ref{palt}) and
its comparison are discussed in the next section.

The next result shows that the preconditioned approach is also more
preferable even for group monotone matrices.

\begin{theorem}\label{pcomp}
Let $A = K-L$ be a proper G-weak regular splitting of a group
monotone matrix $A$. Assume that there exists a nonsingular matrix
$Q$ such that $QA = AQ$ and $A^{\#}Q^{-1}\geq 0.$ If $QA = K_q-L_q$
is a proper G-regular splitting of $QA$ and $QK_q^{\#}\geq K^{\#},$
then $\rho(K_q^{\#}L_q)\leq \rho(K^{\#}L)<1.$
\end{theorem}
\begin{proof}
By Theorem \ref{3.2} and Theorem \ref{3.1}, we have
$\rho(K_q^{\#}L_q)<1$ and $\rho(K^{\#}L)<1$, respectively. Since
$L_qK_q^{\#}\geq 0$, there exists an eigenvector $x\geq 0$ such that
\begin{equation}\label{eq14}
L_qK_q^{\#}x = \rho(K_q^{\#}L_q)x,
\end{equation}
by Theorem \ref{2.1}. This implies $x\in R(L_qK_q^{\#}) \subseteq
R(A)$. Since $N(L_q)\supseteq N(K_q),$ so $QA = (I-L_qK_q^{\#})K_q.$
Therefore,  $A^{\#} = QK_q^{\#}(I-L_qK_q^{\#})^{-1}$ by Theorem
\ref{2.5} (e). Now $QK_q^{\#}\geq K^{\#}$  implies
$A^{\#}(I-L_qK_q^{\#})\geq (I-K^{\#}L)A^{\#}$. Further
simplification yields
\begin{equation}\label{eq15}
A^{\#}L_qK_q^{\#}\leq K^{\#}LA^{\#}.
\end{equation}
Post-multiplying $x$ in (\ref{eq15}) and using equation
(\ref{eq14}), we have $\rho(K_q^{\#}L_q)A^{\#}x \leq
K^{\#}LA^{\#}x$. Let us assume $z = A^{\#}x$. We obtain $z\geq 0$
and $z\neq 0$ ($z = 0$ leads to $x\in R(A)\cap N(A)$ which is not
possible). Hence, by Theorem \ref{2.4} (i) the required result
follows.
\end{proof}

\section{Numerical Examples}
In this section, we discuss a few examples and its numerical
implementation for the proposed theory in the previous section. The
performance measures calculated are the number of iterations (IT),
the mean processing time in seconds (MT) and the estimation of error
bounds. All the numerical examples are worked out by using
Mathematica 10.0 (for examples 5.1-5.5) and MATLAB R2017a (for
examples 5.6-5.7) on an Intel(R) Core(TM)i5, 2.5GHz, 4GBRAM, which
runs on the operating system: Mac OS X El Capitan Version 10.11.6.
We use the following stopping criterion to terminate the process:
The iteration is terminated if $\|x_k-x_{k-1}\|_2 \leq \epsilon$ or
it reaches to the maximum allowed iterations $2000$.
\begin{example}\label{ex5.1}
Let us consider a linear system $Ax=b$ with $A =
  \begin{bmatrix}
   3&1&2\\
    1&-12&13\\
    2&13&-11
  \end{bmatrix}$ and $b=(1,1,0)^{t}.$
Then $A^{\#} =
\begin{bmatrix}
    0.1471&0.0691&0.0781\\
    0.0691&0.0120&0.0571\\
    0.0781&0.0571&0.0210
  \end{bmatrix}
\geq 0 $. Consider the following three proper G-weak regular
splittings of $A$ as
\begin{eqnarray*}
A &=&
\begin{bmatrix}
    4.75&2.5&2.25\\
    1.5833&-11.5&13.0833\\
    3.1667&14&-10.8333
\end{bmatrix}
  -\begin{bmatrix}
    1.75&1.5&0.25\\
    0.5833&0.5&0.0833\\
    1.1667&1&0.1667
    \end{bmatrix}=K-L= ~\mbox{ (Splitting 1)}\\
  &=&
  \begin{bmatrix}
   5&2&3\\
    2&-12&14\\
    3&14&-11
  \end{bmatrix}
  -\begin{bmatrix}
   2&1&1\\
    1&0&1\\
    1&1&0
  \end{bmatrix}=U-V=~\mbox{ (Splitting 2)}\\
  &=&\begin{bmatrix}
    5.2083&2.9583&2.25\\
    2.25&-10.8333&13.0833\\
    2.9583&13.7917&-10.8333
     \end{bmatrix}
  -\begin{bmatrix}
    2.2083&1.9583&0.25\\
    1.25&1.1667&0.0833\\
    0.9583&0.7917&0.1667
   \end{bmatrix}=X-Y=~\mbox{ (Splitting 3)}
\end{eqnarray*}

\begin{equation*}
 \mbox{  with~ }
K^{\#} =
\begin{bmatrix}
    0.0937&0.0476&0.0462\\
    0.0424&0.0013&0.0411\\
    0.0514&0.0463&0.0051
  \end{bmatrix}
\geq 0,~ K^{\#}L =
\begin{bmatrix}
    0.2456&0.2105&0.0351\\
    0.1228&0.1053&0.0175\\
    0.1228&0.1053&0.0175
  \end{bmatrix}
\geq 0,
\end{equation*}

\begin{equation*}
  U^{\#} =
\begin{bmatrix}
    0.0885&0.0417&0.0469\\
    0.0417&0&0.0417\\
    0.0469&0.0417&0.0052
  \end{bmatrix}
\geq 0,~ U^{\#}V =
\begin{bmatrix}
    0.2656&0.1354&0.1302\\
    0.1250&0.0833&0.0417\\
    0.1406&0.0521&0.0885
  \end{bmatrix}
\geq 0,
\end{equation*}

\begin{equation*}
  X^{\#} =
\begin{bmatrix}
    0.0855&0.0446&0.0409\\
    0.0409&0.0007&0.0401\\
    0.0446&0.0439&0.0007
  \end{bmatrix}
\geq 0~ \mbox{and}~ X^{\#}Y =
\begin{bmatrix}
    0.2838&0.2519&0.0319\\
    0.1297&0.1127&0.0170\\
    0.1541&0.1392&0.0149
  \end{bmatrix}
\geq 0.
\end{equation*}
Therefore, $\rho(H)=0.0614\leq $
$\min\{\rho(K^{\#}L)=0.3684,\rho(U^{\#}V)=0.3983,\rho(X^{\#}Y)=0.4163\}<1$.
The numerical results for the convergence analysis is provided in
Table  and comparison results discussed in Table .
\end{example}

The next example shows the importance of the study of the
alternating iteration scheme in the group inverse setting. Note that
existing theory in the literature uses the non-negativity of the
Moore-Penrose inverse, see \cite{mc, mish, misalt} which fails here.

\begin{example}\label{ex5.2}
Let us consider another system $Ax = b$ with $A =
  \begin{bmatrix}
    10&-4&17\\
    54&-42&77\\
   -12&15&-13
  \end{bmatrix}$
   and $b=(-1,-11,4)^{t}.$ The matrix $A$ has non-negative group inverse but does not have non-negative Moore-Penrose inverse. Since
 \begin{equation*}
  A^{\dag} =
  \begin{bmatrix}
   -0.0028&0.0043&-0.0064\\
    0.0577&0.0033&0.0849\\
    0.0363&0.0109&0.0489
  \end{bmatrix}
  \not\geq 0 ~\mbox{but}~
  A^{\#} =
  \begin{bmatrix}
    0.0242&0.0113&0.0565\\
    0.0548&0.0102&0.1164\\
    0.0090&0.0119&0.0265
  \end{bmatrix}
  \geq 0.
  \end{equation*}
Now consider $A$ as $A=K-L=U-V=X-Y,$ where
  $K=\begin{bmatrix}
    14.8681&-0.1590&29.4750\\
    73.5383&-42.9540&115.1930\\
   -14.4671&21.2385&-13.3840
  \end{bmatrix},$
  \begin{equation*}
  U =
  \begin{bmatrix}
    16.1942&-0.9500&31.5405\\
    76.0650&-35.7555&125.4440\\
   -13.7411&16.4528&-15.4113
  \end{bmatrix},~
   X =
  \begin{bmatrix}
    16.9186&-2.1315&32.1250\\
    76.7119&-39.0705&124.3265\\
   -12.9780&16.3380&-13.9758
  \end{bmatrix}.
  \end{equation*}
  Clearly,
  \begin{equation*}
  K^{\#} =
  \begin{bmatrix}
     0.0098&0.0049&0.0230\\
     0.0269&0.0012&0.0544\\
     0.0012&0.0068&0.0073
  \end{bmatrix}\geq 0,~
  K^{\#}L =
  \begin{bmatrix}
     0.0874&0.1763&0.3018\\
     0.0197&0.4414&0.3595\\
     0.1212&0.0438&0.2729
  \end{bmatrix}  \geq 0,
  \end{equation*}

\begin{equation*}
 U^{\#} =
  \begin{bmatrix}
     0.0122&0.0047&0.0277\\
     0.0357&0.0015&0.0722\\
     0.0004&0.0063&0.0054
  \end{bmatrix}\geq 0,~
  U^{\#}V =
  \begin{bmatrix}
     0.1313&0.1067&0.3388\\
     0.1283&0.2230&0.4170\\
     0.1328&0.0486&0.2996
  \end{bmatrix}\geq 0,
  \end{equation*}

 \begin{equation*}
 X^{\#} =
  \begin{bmatrix}
     0.0129&0.0043&0.0288\\
     0.0365&0.0004&0.0730\\
     0.0011&0.0063&0.0068
  \end{bmatrix}\geq 0,~
  X^{\#}Y =
  \begin{bmatrix}
     0.1593&0.0754&0.3718\\
     0.1899&0.1670&0.4992\\
     0.1440&0.0296&0.3081
  \end{bmatrix}\geq 0.
  \end{equation*}
  The spectral radius of the individual splittings and the alternating iteration matrix are as follows: $\rho(K^{\#}L) = 0.5841$, $\rho(U^{\#}V) = 0.5515$, $\rho(X^{\#}Y) = 0.5541$,
 $\rho(H) = 0.1728<1.$  The convergence analysis is provided in the Table \ref{tab:table1}.
 \end{example}

\begin{example}\label{ex5.3}
Let us consider a system $Ax = b$ with $A =
  \begin{bmatrix}
    3&-1&-9\\
   -5&-5&-12\\
   -18&-14&-27
  \end{bmatrix}$
   and $b=(-18,-4,6)^{t}.$ Here $A^{\#} =
  \begin{bmatrix}
     0.0972&0.0294&-0.0413\\
    0.0099&0.0007&-0.0137\\
    -0.0674&-0.0274&0.0001
  \end{bmatrix}\ngeq 0.$ We find a nonsingular matrix $$Q = \begin{bmatrix}
     4.9000&-1.8600&-0.5300\\
    -2.0371&7.5984&-2.8996\\
    -1.8670&-2.7776&0.9755
  \end{bmatrix}$$ such that with  $QA = AQ$ and $(QA)^{\#} \geq 0. $
  Now consider the new system $QAx = Qb$, where $QA$ has three different proper G-weak regular splittings i.e., $QA = K_q-L_q = U_q-V_q = X_q-Y_q$. Now
  $K_q=\begin{bmatrix}
    36.0660&12.5447&-8.7030\\
    9.8863&6.1737&8.6910\\
   -6.4071&5.9764&34.7760
  \end{bmatrix},$
  \begin{equation*}
  U_q =
  \begin{bmatrix}
    35.6316&12.4668&-8.3015\\
    9.4460&5.8062&7.9290\\
   -7.2936&4.9516&32.0885
   \end{bmatrix},~
   X_q =
  \begin{bmatrix}
    34.9083&12.2843&-7.8472\\
    8.7488&5.3427&7.2025\\
   -8.6617&3.7439&29.4545
  \end{bmatrix}.
  \end{equation*}
  Clearly,
  \begin{equation*}
  K_q^{\#} =
  \begin{bmatrix}
     0.0222&0.0091&0.0001\\
     0.0077&0.0052&0.0084\\
     0.0008&0.0066&0.0254
    \end{bmatrix}\geq 0,~
  K_q^{\#}L_q=
  \begin{bmatrix}
    0.0725&0.0303&0.0030\\
    0.0530&0.0465&0.1007\\
    0.0866&0.1091&0.2990
   \end{bmatrix}  \geq 0,
   \end{equation*}

\begin{equation*}
 U_q^{\#} =
  \begin{bmatrix}
     0.0225&0.0092&0.0001\\
     0.0081&0.0056&0.0092\\
     0.0018&0.0075&0.0277
    \end{bmatrix}\geq 0,~
  U_q^{\#}V_q =
  \begin{bmatrix}
    0.0597&0.0255&0.0048\\
    0.0428&0.0381&0.0838\\
    0.0686&0.0889&0.2466
  \end{bmatrix}\geq 0,
  \end{equation*}

 \begin{equation*}
 X_q^{\#} =
  \begin{bmatrix}
    0.0231&0.0094&0.0001\\
    0.0088&0.0061&0.0102\\
    0.0034&0.0089&0.0304
    \end{bmatrix}\geq 0,~
  X_q^{\#}Y_q =
  \begin{bmatrix}
    0.0379&0.0175&0.0083\\
    0.0223&0.0251&0.0649\\
    0.0290&0.0578&0.1863
  \end{bmatrix}\geq 0.
  \end{equation*}
  The spectral radius of the individual splittings and the alternating iteration matrix are as follows: $\rho(K_q^{\#}L_q) = 0.3417$, $\rho(U_q^{\#}V_q) = 0.2823$, $\rho(X_q^{\#}Y_q) = 0.2097$,
 $\rho(H) = 0.0203<1.$  The convergence analysis  is shown in the Table .
 \end{example}

\begin{example}\label{ex5.4}
Let us consider a system $Ax=b$ with $A =
  \begin{bmatrix}
   47& -9& -5\\
    5& 0& 4\\
    -14& 3& 3
\end{bmatrix}$ and $b=(6,3,-1)^{t}.$
Clearly $A$ is group monotone matrix since $A^{\#} =
\begin{bmatrix}
    0.0843& 0.0311& 0.2145\\
    0.2801& 0.1526& 0.9462\\
    0.0653& 0.0405& 0.2439
\end{bmatrix}
\geq 0 $. Let
\begin{equation*}
A = K - L =
 \begin{bmatrix}
   52.2707& -9.3666& -2.5190\\
   7.1598& 1.8711& 14.5844\\
  -15.0370& 3.7459& 5.7011
\end{bmatrix}
-
\begin{bmatrix}
   5.2707& -0.3666& 2.4810\\
   2.1598& 1.8711& 10.5844\\
   -1.0370& 0.7459& 2.7011
\end{bmatrix}
\end{equation*}
is a proper G-weak regular splitting of $A$ since
\begin{equation*}
K^{\#} =
    \begin{bmatrix}
    0.0380& 0.0065& 0.0610\\
    0.0884& 0.0449& 0.2834\\
    0.0168& 0.0128& 0.0741
    \end{bmatrix}
    \geq 0~\mbox{and}~ K^{\#}L =
    \begin{bmatrix}
    0.1510& 0.0436& 0.3274\\
    0.2693& 0.2630& 1.4604\\
    0.0394& 0.0731& 0.3777
    \end{bmatrix}
    \geq 0.
    \end{equation*}
 Then there exists a nonsingular matrix
\begin{equation*}
Q =
    \begin{bmatrix}
     12.1426& 2.4576& 7.0308\\
     7.1770& 22.2770& 26.4823\\
     4.3098& 0.6414& 24.3790
    \end{bmatrix} ,~ QA = \begin{bmatrix}
     484.5610& -88.1914& -29.7904\\
     77.9525& 14.8537& 132.6700\\
     -135.536& 34.3484& 54.1534
    \end{bmatrix}
    \end{equation*}
    such that $QA = AQ$ and $(QA)^{\#} =
    \begin{bmatrix}
    0.0041& 0.0007& 0.0068\\
    0.0092& 0.0049& 0.0308\\
    0.0017& 0.0014& 0.0080
   \end{bmatrix}
    \geq 0.
$ Now let us consider a new system $QAx = Qb,$ where $QA$ has a
proper G-regular splitting i.e., $QA = K_q-L_q$. Now
\begin{equation*}
K_q =
\begin{bmatrix}
    493.1640& -76.7488& 31.2533\\
    89.3695& 28.7235& 207.4530\\
   -134.5980& 35.1574& 58.7334
\end{bmatrix}~\mbox{and}~L_q =
\begin{bmatrix}
    8.6028& 11.4425& 61.0437\\
    11.4170& 13.8697& 74.7837\\
    0.9386& 0.8091& 4.5800
\end{bmatrix}.
\end{equation*}

\begin{equation*}
K_q^{\#} =
\begin{bmatrix}
     0.0032& 0.0002& 0.0035\\
     0.0063& 0.0032& 0.0202\\
     0.0010& 0.0010& 0.0056
    \end{bmatrix} \geq 0 ~ \mbox{and}~ L_q \geq 0.
\end{equation*}
Since $QK_q^{\#}-K^{\#} =
\begin{bmatrix}
 0.0235& 0.0109& 0.0703\\
 0.1029& 0.0543& 0.3393\\
 0.0265& 0.0145& 0.0897
\end{bmatrix}
\geq 0.$ Therefore, $\rho(K_q^{\#}L_q) = 0.3318 \leq 0.6993 =
\rho(K^{\#}L)<1.$ The numerical results for comparison is discussed
in Table.
\end{example}
The concept of the three-step alternating iteration scheme and
comparison results can be applied to nonsingular system.  The
validation of the proposed approach is explained in the next
example.

\begin{example}\label{ex5.5}
Let us consider the nonsingular $M$-matrix
\begin{equation*}
    A =  \begin{bmatrix}
    10.8654&-0.3333&-1.4444&-1.2222&-0.6667&-0.1111&-1.3333&-2&-0.5556\\
    -1.6667&9.0877&-2&-1.3333&-0.8889&-2&-0.2222&-0.5556&-0.3333\\
    -1.6667&-1.5556&9.8654&-1.1111&-1.2222&-1.3333&-1.5556&-1.8889&-2.2222\\
    -0.7778&-0.8889&-2.2222&9.1988&-1.8889&-1&-0.1111&-0.5556&-0.5556\\
    -1.4444&-0.4444&-1.2222&-1.2222&10.6432&-0.1111&-0.1111&-1.8889&-2.1111\\
    -1.5556&-0.5556&-0.4444&-0.3333&-1.7778&9.9765&-1.5556&-1.1111&-2\\
    -1.8889&-1.1111&-0.3333&-0.5556&-2.1111&-1.5556&9.6432&-1.8889&-2.1111\\
    -0.8889&-2&-0.1111&-0.1111&-0.5556&-0.3333&-0.2222&10.8654&-0.3333\\
    -0.6667&-0.6667&-1.3333&-1.4444&-1.5556&-1.4444&-1.6667&-2.2222&9.5321
     \end{bmatrix}
    \end{equation*}
    $ ~~ = K-L = U-V = X-Y$ are three weak regular splitting of $A$ such that $A^{-1}\geq 0.$ In this case
    \begin{equation*}
    K =  \begin{bmatrix}
    11.1529&0.0167&-1.5694&-1.0347&-0.3042&0.3014&-1.2333&-1.6875&-0.5931\\
    -1.3417&8.9502&-1.7125&-1.1333&-0.4639&-2.2750&0.1278&-0.7556&-0.7083\\
    -1.7292&-1.2306&10.0779&-1.2611&-1.0097&-1.5833&-1.4556&-1.9889&-2.2722\\
    -0.5778&-1.0139&-2.2972&9.2488&-1.4139&-1&-0.4111&-0.3806&-0.2681\\
   -1.8069&-0.0694&-1.3347&-1.2597&10.7057&0.2889&-0.0986&-1.4389&-1.9986\\
    -1.1431&-0.7056&-0.1819&-0.1458&-1.3653&9.8390&-1.1056&-1.2236&-1.6375\\
    -1.7139&-0.8986&-0.0208&-0.1806&-2.4361&-1.2056&9.8182&-1.7514&-2.2486\\
    -0.4889&-1.95&-0.2361&-0.0236&-0.2431&0.1292&-0.0972&11.1279&0.0417\\
    -0.2792&-0.3542&-1.0458&-1.7069&-1.7181&-1.0819&-1.8167&-1.9347&9.7696
     \end{bmatrix},
    \end{equation*}

    \begin{equation*}
    U =  \begin{bmatrix}
    11.0404&-0.0458&-1.3569&-1.4097&-0.5542&0.0389&-1.5333&-1.7125&-1.0181\\
    -1.1667&9.1002&-1.9250&-0.8958&-0.7389&-2.2750&-0.3097&-0.3681&-0.1083\\
    -1.1792&-1.5806&9.5654&-0.8736&-0.9972&-0.9833&-1.3806&-1.7764&-2.0347\\
    -1.0653&-0.6639&-1.8472&9.1113&-1.6139&-0.8250&0.0014&-0.5806&-0.1681\\
   -1.6944&-0.1319&-1.0097&-1.1597&10.6557&0.1139&0.1014&-1.6139&-1.8986\\
    -1.2181&-0.8931&-0.0319&-0.2458&-1.5903&10.2015&-1.6181&-0.8486&-1.5\\
    -1.5014&-0.9986&-0.3458&-0.2056&-2.1111&-1.5806&10.1432&-1.6264&   -1.6236\\
    -0.6514&-1.5375&-0.1611&0.2139&-0.0556&-0.1458&0.1778&11.2529&0.0917\\
    -0.8292&-0.7042&-0.8833&-1.4694&-1.7056&-1.1319&-1.4167&-2.1722&9.1446
     \end{bmatrix},
    \end{equation*}

     \begin{equation*}
    X =  \begin{bmatrix}
    11.1779&-0.0833&-1.6069&-1.3972&-1.1167&0.3764&-0.8333&-1.5875&-0.8556\\
    -2.1292&9.5502&-2.0250&-1.3833&-0.5389&-1.7375&0.2778&-0.1806&0.0667\\
    -1.2417&-1.4681&9.9029&-1.2236&-1.2347&-0.8958&-1.9431&-2.1639&-2.3222\\
   -0.5278&-1.1139&-2.4097&9.6363&-1.8264&-1.0750&-0.2611&-0.5806&-0.4931\\
   -1.0319&-0.1319&-0.9597&-0.7347&11.0307&0.0264&0.2514&-1.7139&-2.4361\\
   -1.3556&-0.5181&-0.2694&-0.6708&-1.6278&9.6890&-1.3931&-0.6611&-1.9250\\
   -2.1389&-1.0111&-0.1833&-0.0931&-1.8361&-1.5181&9.5807&-1.8014&-1.8236\\
    -0.7014&-2.2125&-0.1236&0.2764&-0.0681&-0.2458&-0.4472&10.7029&0.0292\\
    -0.9417&-0.7667&-1.0708&-1.2569&-1.7056&-1.0319&-1.5792&-1.9347&10.0071
     \end{bmatrix}.
    \end{equation*}
    Here $\rho(X^{-1}YU^{-1}VK^{-1}L) = 0.1513 \leq 0.3038 = \rho(U^{-1}VK^{-1}L) \leq 0.5346 = \rho(K^{-1}L).$
The comparison analysis of one step, two-step and three-step
alternating iteration scheme is provided in the table. It also
contains the same analysis for  two random nonsingular matrices of
order $1000$ and $2000$.
\end{example}

\section{Conclusion}
We have introduced the three-step alternating iterations for
singular linear systems of index 1 and studied its convergence
criteria. Three algorithms are also provided for numerical
computation and complexity. Finally, a comparison result is proved
which guarantees the fact that the three-step alternating iterations
converges faster than the usual one, and is also shown through
examples. The authors of \cite{mc}, \cite{mish} and \cite{misalt}
studied the two-step alternating iterations for rectangular matrices
using the Moore-Penrose inverses, very recently. However,  their
works lack computational implementation which is addressed in this
paper. Finally, we conclude the paper with the comparative analysis
of one step, two step and three step iterations.

The iterative methods (i.e., matrix splitting methods) and
semi-iterative methods are among many methods that have been
suggested in the literature to solve real singular linear systems. A
matrix is called an {\it EP matrix} if $R(A)=R(A^t)$. If $A$ is an
EP matrix, then the proposed scheme will converge to the least
squares solution of minimum norm. Migall{\'o}n {\it et al.}
\cite{miga} studied alternating two-stage methods for consistent
linear systems to obtain the parallel solution of Markov chains,
recently. The same authors further extended the same notion in
\cite{miga1}. Further applications of this theory to compute the
PageRank of a google matrix can also be found in the recent article
\cite{gu}. Hence, we conclude this article with the hope that our
work may help to deal with singular linear systems which appear in
different areas of mathematics as mentioned above and in the
introduction part. We hope that this work will provide useful
insights into extending this approach and thus help in solving
rectangular linear systems in a faster way.

\end{document}